\documentclass[11pt]{amsart}
\usepackage{graphicx}
\usepackage{amsmath}
\usepackage{amscd}
\usepackage{amssymb}
\usepackage{latexsym}
\usepackage{epstopdf}
\newcommand{\Q}{\mathbb{Q}}
\newcommand{\Z}{\mathbb{Z}}
\newcommand{\C}{\mathbb{C}}

\newcommand{\set}[1]{\{ #1 \}}

\newcommand{\dbar}{\overline{\partial}}
\newcommand{\twopi}[1]{\frac{ #1 }{2\pi i}}
\newcommand{\Proj}{\mathbb{P}}
\newcommand{\Bl}[1]{\widetilde{#1}}

\newtheorem{thm}{Theorem}
\newtheorem{prop}{Proposition}
\newtheorem{lem}{Lemma}
\newtheorem{ex}{Example}
\newtheorem{defn}{Definition}
\newtheorem{rmk}{Remark}
\newtheorem*{ack*}{Acknowledgements}

\title{Singular Elliptic Genus of Normal Surfaces}
\author{Robert Waelder}

\begin{document}
\begin{abstract}
We define the singular elliptic genus for arbitrary normal surfaces, prove that it is a birational invariant, and show that it generalizes the singular elliptic genus of Borisov and Libgober and the stringy $\chi_y$ genus of Batyrev and Veys.
\end{abstract}
\email{rwaelder@math.ucla.edu}
\maketitle

\section{Introduction}
The generalization of smooth invariants to singular varieties has a number of interesting applications. One example is the topological mirror symmetry test, which asserts that mirror pairs $Z$ and $\hat{Z}$ must have mirror Hodge numbers in the sense that $h^{p,q}(Z) = h^{n-p,q}(\hat{Z})$. When either $Z$ or $\hat{Z}$ is singular, it turns out that one has to replace ordinary Hodge numbers with the stringy Hodge numbers of Batyrev. For $Z$ a $\Q$-Gorenstein variety with log-terminal singularties, Batyrev defines the stringy Hodge numbers of $Z$ as the coefficients of $u^p v^q$ of the stringy $E$-function, constructed as follows: Take $f:X\rightarrow Z$ to be a log resolution of singularities with $K_X = f^*K_Z + \sum_I a_i D_i$. Then $E_{str}(Z;u,v) =$
\begin{align}\label{E string}
\sum_{J\subset I}E(D_{J}^{o};u,v)\prod_{J}\frac{uv-1}{(uv)^{a_i+1}-1}.
\end{align}
Here $D_{J}^{o} = \bigcap_J D_j\backslash \bigcup_{J^c} D_i$, and $E(D_{J}^{o};u,v)$ denotes the ordinary $E$-function. Log-terminality means that $a_i > -1$. In particular, we need not worry about dividing by zero in the above expression. Batyrev's stringy $E$-function has an elegant interpretation in terms of motivic integration, and its independence on the choice of a log resolution ultimately follows from the change of variable formula for motivic integrals. In addition to providing the proper context for the topological mirror symmetry test, the stringy $E$-function is an important tool in Batyrev's proof of the McKay correspondence \cite{B}. 

Along similar lines, Borisov, Libgober, and Chin-Lung Wang have generalized the complex elliptic genus to $\Q$-Gorenstein varieties with log-terminal singularities. With $X$ and $Z$ defined as above, the singular elliptic genus of $Z$ is defined by the formula:
\begin{align}\label{BL ellip genus}
Ell(Z;z,\tau)=\int_X\prod_{TX}\frac{{x_i}\vartheta(\twopi{x_i}-z)}
{\vartheta(\twopi{x_i})}\prod_{k\in I}
\frac{\vartheta(\twopi{D_k}-(a_k+1) z)\vartheta(z)}
{\vartheta(\twopi{D_k}-z)\vartheta((a_k+1) z)}
\end{align}
Here $x_i$ denote the formal Chern roots of the holomorphic tangent bundle $TX$. The independence of $Ell(Z;z,\tau)$ on the choice of a log resolution follows from the weak factorization theorem of Wlodarczyk \cite{W}. This generalization of the elliptic genus plays an essential role in Chin-Lung Wang's proof that $K$-equivalent varieties have the same elliptic genus \cite{CLW}, and in Borisov and Libgober's proof of the McKay correspondence for elliptic genera \cite{BL}. 

It is interesting to note that the above two approaches to defining smooth invariants of singular varieties, while coming from very different viewpoints, nevertheless fail for the same class of singularities, i.e., for the non-log-terminal ones. This is for the simple reason that if we allowed a discrepancy coefficient $a_k$ to be $-1$, we would have to divide by zero in either equations (\ref{E string}) or (\ref{BL ellip genus}). This paper began as an attempt to overcome this obstacle. In particular, we will show that we can define the singular elliptic genus for arbitrary normal surfaces. Partial work has been done in this area by Willem Veys \cite{Veys}, who extended Batyrev's stringy $E$-function for normal surfaces without purely log-canonical singularities. The singular elliptic genus defined in this paper generalizes Veys' stringy $\chi_y$ genus in this setting, as well as the singular elliptic genus of Borisov, Libgober, and Chin-Lung Wang. 
\begin{ack*}\rm
Much of the research for this paper was conducted during the author's stay in Zhejiang University, Hangzhou. The author wishes to thank Zhejiang University for its generous hospitality.
\end{ack*}
\section{Preliminaries}
\subsection{Theta function identities}
We will make use of the following properties of the Jacobi theta function 
$$\vartheta(t,\tau) = q^{\frac{1}{8}}2\sin\pi t
\prod_{n=1}^{\infty}(1-q^n)\prod_{n=1}^{\infty}(1-q^ne^{2\pi it})
(1-q^n e^{-2\pi it})$$ 
where $q=e^{2\pi i\tau}$. The action of $SL_2(\Z)$ on $\C\times\mathbb{H}$ given by 
\begin{align*}
(t,\tau) \mapsto (\frac{t}{c\tau+d},\frac{a\tau+b}{c\tau+d})
\end{align*}
induces a corresponding action on $\vartheta(t,\tau)$.
With respect to the generators of the $SL_2(\Z)$ action: $(t,\tau)\mapsto (t,\tau+1)$ and $(t,\tau)\mapsto(\frac{t}{\tau},\frac{-1}{\tau})$ we have the following transformation formulas for $\vartheta(t,\tau)$:
\begin{align}
\vartheta(t,\tau+1) &= e^{\pi i/4}\vartheta(t,\tau) \label{mod one}\\ 
\vartheta(\frac{t}{\tau},-\frac{1}{\tau}) &= \frac{1}{i}\sqrt{\frac{\tau}{i}}e^{\pi it^2/\tau}\vartheta(t,\tau) \label{mod two}
\end{align}
We will also make use of the translation rules:
\begin{align*}
\vartheta(t+1,\tau) &= -\vartheta(t,\tau)\\
\vartheta(t+\tau,\tau) &= -q^{-1/2}e^{-2\pi it}\vartheta(t,\tau)
\end{align*}
\subsection{The Elliptic Genus}
The elliptic genus $Ell(X;z,\tau)$ of a smooth complex $n$-manifold $X$ is defined as the index of the operator:
$$\dbar\otimes\bigotimes_{n=1}^{\infty}\Lambda_{-yq^{n-1}}T^{*}X\otimes
\Lambda_{-y^{-1}q^n}TX\otimes S_{q^{n}}T^{*}X\otimes S_{q^n}TX.$$
Here $TX$ is the holomorphic tangent bundle; $\Lambda_t(E)$ and $S_t(E)$ denote the formal sums of exterior and symmetric powers of $tE$; and $y,q$ are defined by $y=e^{2\pi iz}$ and $q=e^{2\pi i\tau}$. By Riemann-Roch, this index is given by the following integral involving the formal chern roots $x_i$ of $TX$:
$$\int_X \prod_{TX}\frac{x_i\vartheta(\twopi{x_i}-z,\tau)}{\vartheta(\twopi{x_i},\tau)}$$
One sees from the above expression that special values of the elliptic genus produce many interesting geometric invariants. For instance, $Ell(X;z,q=0) = y^{-n/2}\chi_{-y}(X)$, where $\chi_{-y}$ is the Hirzebruch $\chi_{-y}$ genus, and $Ell(X;\frac{1}{2},\tau)$ reproduces the signature.

If $D=\sum a_i D_i$ is a smooth divisor with simple normal crossings and coefficients $a_i\neq -1$, Borisov, Libgober, and Chin-Lung Wang have extended the notion of elliptic genus to that of the pair $(X,D)$ by setting $Ell(X,D;z,\tau)$ equal to the expression on the RHS of (\ref{BL ellip genus}) \cite{BLSing,CLW}. 

\subsection{Definitions from singularity theory}
We say that a singular variety $Z$ is Gorenstein (resp. $\Q$-Gorenstein) if $K_Z$ is Cartier (resp. $\Q$-Cartier). In general, this technical assumption is required in order to make sense of the pull-back of the canonical class of $Z$. However, if $Z$ is a normal surface, we can always make sense of the pull-back of $K_Z$ essentially using the adjunction formula.

Under these assumptions, let $f:X \rightarrow Z$ be a resolution of singularities. We say that $X$ is a log resolution if the exceptional components of the resolution are smooth divisors with simple normal crossings. Writing $K_X = f^*K_Z + \sum a_i D_i$, we say that the singularities of $Z$ are log-terminal (resp. log-canonical) if the discrepancy coefficients $a_i$ satisfy $a_i > -1$ (resp. $a_i \geq -1$). We say that the singularities of $Z$ are purely log-canonical if they are log-canonical and not log-terminal.

\section{Singular Elliptic Genus}
\subsection{Definitions and Notation}
We begin by setting up some notation. Let $X$ be a smooth surface and $C = \sum a_i C_i$ a smooth divisor on $X$ with simple normal crossings. We define a labeled graph $\Gamma(C)$ as follows: For each component $C_i$, we draw a vertex $v_i$. Whenever $C_i$ intersects $C_j$, we draw $C_iC_j$ edges connecting $v_i$ to $v_j$. Finally, we assign the label $(a_i,m_i,g_i)$ to the vertex $v_i$, where $a_i$ is the coefficient of $C_i$ in $C$, $m_i = -C_iC_i$, and $g_i$ is the genus of $C_i$. If $v_i$ has label $(-1,m_i,g_i)$, we refer to $v_i$ as a $-1$ \it{vertex}\rm. If a $-1$ vertex $v$ with genus $0$ is connected to exactly $2$ vertices $v_1$ and $v_2$ such that $a_1+a_2 = -2$, or if $v$ is connected to exactly one vertex $v_i$ with $a_i = -2$, we refer to $v$ as a \it{bridge}\rm. 

For $v_i$ a vertex with $a_i > -1$ (resp. $a_i < -1$), we say that $v_i$ is connected to a $-1$ vertex $v_j$ if there exists a path connecting $v_i$ to $v_j$ such that every vertex $v_k \neq v_j$ on the path has coefficient $a_k \geq a_i$ (resp. $a_k \leq a_i$) and genus $0$. For each coefficient $a\neq -1$ of a component of $C$, let $S_a$ denote the collection of vertices with label $(a,\cdot,\cdot)$ connected to a $-1$ vertex which is not a bridge. Let $R_a$ denote the collection of vertices in $S_a$ which are in addition connected to a bridge. Let $B_{a}$ denote the collection of bridges connecting vertices with labels $a$ and $-2-a$. Finally, let $B^{'}_a \subset B_a$ denote the collection of bridges connected to a vertex in $R_a$.

\begin{defn}\label{Def}
Let $C$ be a smooth divisor with simple normal crossings on $X$. Let $d$ denote the number of edges connecting two $-1$ vertices on the resolution graph. We define the elliptic genus $Ell(X,C;z,\tau)=$ 
\begin{align*}
&\int_X\prod_{i=1}^{2}\frac{{x_i}\vartheta(\twopi{x_i}-z)}
{\vartheta(\twopi{x_i})}
\prod_{a_j=-1}\frac{\vartheta(\twopi{C_j}+2z)\vartheta(z)}
{\vartheta(\twopi{C_j}+z)\vartheta(2 z)}
\prod_{a_k\neq-1}\frac{\vartheta(\twopi{C_k}-(a_k+1) z)\vartheta(z)}
{\vartheta(\twopi{C_k}-z)\vartheta((a_k+1) z)}+\\
%&\sum_{a}|S_a \backslash R_a|\frac{\vartheta((a+2)z)\vartheta(a z)}{\vartheta((a+1)z)^2}+\sum_{a}(|R_a|-\sum_{v_i\in B^{'}_a} m_i-1)\frac{\vartheta((a+2)z)\vartheta(a z)}{\vartheta((a+1)z)^2}\\
%-\sum_{a_i = -1}Ell(C_i,\sum_{v_k\in\Gamma_{v_i}}a_k C_k)\frac{\vartheta(z)}{\vartheta(2z)}\\
&\sum_{a}(|S_a|-
\sum_{v_i\in B^{'}_a} m_i-1)\frac{\vartheta((a+2)z)\vartheta(a z)}{\vartheta((a+1)z)^2}+\\
&\sum_{a, v_i\in B_a}m_i\frac{\vartheta((a+2)z)\vartheta(a z)}{\vartheta((a+1)z)^2}+d\frac{\vartheta(3z)\vartheta(z)}{\vartheta(2z)^2}.
\end{align*}
\end{defn}
As in the above formula, we often write $\vartheta(t)$ for $\vartheta(t,\tau)$. For convenience, we also define for a pair $(X,C)$ the ``naive" elliptic genus $Ell_{nv}(X,C;z,\tau) =$
\begin{align*}
&\int_X\prod_{i=1}^{2}\frac{{x_i}\vartheta(\twopi{x_i}-z)}
{\vartheta(\twopi{x_i})}
\prod_{a_j=-1}\frac{\vartheta(\twopi{C_j}+2z)\vartheta(z)}
{\vartheta(\twopi{C_j}+z)\vartheta(2 z)}
\prod_{a_k\neq-1}\frac{\vartheta(\twopi{C_k}-(a_k+1) z)\vartheta(z)}
{\vartheta(\twopi{C_k}-z)\vartheta((a_k+1) z)}.
\end{align*}

Suppose that $Z$ is a normal surface, and let $f:X\rightarrow Z$ be a log resolution with $K_X = f^*K_Z + C$. We define the singular elliptic genus of $Z$ to be $Ell(X,C;z,\tau)$. More generally if $K_Z-\Delta$ has a well-defined pull-back, we define $Ell(Z,\Delta;z,\tau) = Ell(X,C;z,\tau)$ where this time $C$ is defined by $K_X = f^*(K_Z-\Delta) +C$. We will show in section \ref{birat inv} that $Ell(Z,\Delta;z,\tau)$ is independent of the choice of log resolution.

One could also define the singular elliptic genus of a normal surface $Z$ as follows: Choose a resolution $f:X\rightarrow Z$ with exceptional divisor $C = \sum a_i C_i$ and 
introduce a perturbation $C_{\varepsilon} = \sum\varepsilon b_iC_i$ so that $C+C_\varepsilon$ has no $-1$ coefficients. Define $\widehat{Ell}(Z;z,\tau)$ to be the limit as $\varepsilon\to 0$ of $Ell(X,C+C_\varepsilon;z,\tau)$ if the limit exists, where in this case $Ell(X,C+C_{\varepsilon};z,\tau)$ is the elliptic genus of $(X,C+C_\varepsilon)$ introduced by Borisov, Libgober, and Chin-Lung Wang. One way to
introduce such a perturbation is let $H$ be an ample divisor crossing through the singularities of $Z$, and define $C_\varepsilon$ by the relation $K_X = f^*(K_Z+\varepsilon H)+ C+C_\varepsilon$. This is the approach taken by Borisov and Libgober in \cite{BLSing}. In general, $\widehat{Ell}(Z;z,\tau)$ depends on the choice of the perturbation. In what follows, will show that for all normal singularities which are not purely log-canonical, and also for simple elliptic singularities, we can choose a perturbation so that the limit exists and corresponds to $Ell(Z;z,\tau)$ as given in definition \ref{Def}. Note that by the classification given in \cite{Aster}, these constitute all normal surface singularities except for $6$ types (cusp singularities, and quotients of cusp and elliptic singularities).

\begin{rmk}\rm
We will see that the definition \ref{Def} is chosen to ensure that $Ell(Z,\Delta;z,\tau)$ is invariant under the choice of birational model. We note however that there remains some ambiguity as to the proper way to count the contributions to $Ell(Z,\Delta;z,\tau)$ coming from vertices in $S_a$. For those cases in which we can compare $Ell(Z;z,\tau)$ to other invariants of normal surfaces, $|S_a| = 0$, so this ambiguity disappears.
\end{rmk}

\section{Birational Invariance of the Singular Elliptic Genus}\label{birat inv}
\begin{thm}
For $Z$ a normal surface, let either $\Delta = 0$ or $K_Z-\Delta$ be $\Q$-Cartier. Then $Ell(Z,\Delta;z,\tau)$ is independent of the choice of log resolution.
\end{thm}

\begin{proof}
Let $(X,C) \rightarrow (Z,\Delta)$ be a log resolution. It suffices to prove that if $(\Bl{X},\Bl{C})$ is the blow-up of $(X,C)$ at a point with normal crossings with respect to $C$ then $Ell(X,C;z,\tau) = Ell(\Bl{X},\Bl{C};z,\tau)$. Let $f:\Bl{X}\rightarrow X$ be the blow-up map, with exceptional divisor $E$. Then $T\Bl{X}$ is stably equivalent to $f^*TX\oplus \mathcal{O}(-E)\oplus\mathcal{O}(-E)\oplus\mathcal{O}(E)$ (see \cite{Fulton}).

There are five cases to consider:

CASE $1$: $\Bl{X}$ is the blow up at $C_1\cap C_2$, where $C_1$ and $C_2$ have coefficients equal to $-1$ in $C$. Then $E$ has coefficient equal to $-1$ in $\Bl{C}$. We assume for notational simplicity that $C_1$ and $C_2$ are the only components of $C$. Then $Ell(\Bl{X},\Bl{C};z,\tau)=$
\begin{align*}
&\int_{\Bl{X}}f^*\bigg\{\prod_{i=1}^{2}\frac{{x_i}\vartheta(\twopi{x_i}-z)}
{\vartheta(\twopi{x_i})}\bigg\}
\bigg\{\frac{-E\vartheta(-E-z)\vartheta'(0)}{\vartheta(-E)\vartheta(-z)}\bigg\}^2\frac{E\vartheta(E-z)\vartheta'(0)}{\vartheta(E)\vartheta(-z)}\times\\
&\prod_{i=1}^2\frac{\vartheta(f^*C_i-E+2z)\vartheta(z)}{\vartheta(f^*C_i-E+z)\vartheta(2z)}\times\frac{\vartheta(E+2z)\vartheta(z)}{\vartheta(E+z)\vartheta(2z)} + 2\frac{\vartheta(3z)\vartheta(z)}{\vartheta(2z)^2}.
\end{align*}
Write the above expression as:
\begin{align*}
\int_{\Bl{X}}f^*\bigg\{\prod_{i=1}^{2}\frac{{x_i}\vartheta(\twopi{x_i}-z)}
{\vartheta(\twopi{x_i})}\bigg\}R(E,f^*C_1,f^*C_2)+2\frac{\vartheta(3z)\vartheta(z)}{\vartheta(2z)^2}.
\end{align*}
where $R(E,f^*C_1,f^*C_2) = 1 + R_1(f^*C_1,f^*C_2)E + R_2(f^*C_1,f^*C_2)E^2$. By the projection formula, the above integral evaluates to:
\begin{align*}
\int_{X}\prod_{i=1}^{2}\frac{{x_i}\vartheta(\twopi{x_i}-z)}
{\vartheta(\twopi{x_i})}\prod_{i=1}^2\frac{\vartheta(C_i+2z)\vartheta(z)}{\vartheta(C_i+z)\vartheta(2z)}-R_2(0,0)\frac{\vartheta(-z)^2}{\vartheta'(0)^2}
+2\frac{\vartheta(3z)\vartheta(z)}{\vartheta(2z)^2}.
\end{align*}
If we let $F(t) = t^{-3}R(t,0,0)$, then $R_2(0,0)$ is the residue of $F(t)$ at $t=0$. Note that $F(t) =$
\begin{align*}
\frac{\vartheta(t+z)\vartheta'(0)}{\vartheta(t)\vartheta(z)}
\frac{\vartheta(t+2z)\vartheta'(0)}{\vartheta(t)\vartheta(2z)}
\frac{\vartheta(t-2z)\vartheta'(0)}{\vartheta(t)\vartheta(-2z)}
\frac{\vartheta(t-2z)\vartheta(z)}{\vartheta(t-z)\vartheta(2z)}
%&\bigg\{\frac{\vartheta(t+z)\vartheta'(0)}{\vartheta(t)\vartheta(z)}\bigg\}^2
%\frac{\vartheta(t-z)\vartheta'(0)}{\vartheta(t)\vartheta(-z)}
%\bigg\{\frac{\vartheta(t-2z)\vartheta(z)}{\vartheta(t-z)\vartheta(2z)}\bigg\}^2
%\frac{\vartheta(t+2z)\vartheta(z)}{\vartheta(t+z)\vartheta(2z)}
\end{align*}
Using the transformation properties of the Jacobi theta function, one easily checks that $F(t)$ defines a meromorphic function on a complex torus. Thus, $\mathrm{Res}_{t=0}F(t) = -\mathrm{Res}_{t=z}F(t) = \big\{\frac{\vartheta'(0)}{\vartheta(z)}\big\}^2\frac{\vartheta(3z)\vartheta(z)}{\vartheta(2z)^2}$. We therefore have that $Ell(\Bl{X},\Bl{C};z,\tau)=$
\begin{align*}
\int_{X}\prod_{i=1}^{2}\frac{{x_i}\vartheta(\twopi{x_i}-z)}
{\vartheta(\twopi{x_i})}\prod_{i=1}^2\frac{\vartheta(C_i+2z)\vartheta(z)}{\vartheta(C_i+z)\vartheta(2z)}-\frac{\vartheta(3z)\vartheta(z)}{\vartheta(2z)^2}+2\frac{\vartheta(3z)\vartheta(z)}{\vartheta(2z)^2}
\end{align*}
which equals $Ell(X,C;z,\tau)$. Note that passing from $\Gamma(C)$ to $\Gamma(\Bl{C})$ leaves the $-1$-connectivity properties of $\Gamma(C)$ and the multiplicities of bridges of $\Gamma(C)$ unchanged. This justifies our assumption that $C_1$ and $C_2$ are the only components of $C$.

CASE $2$: $\Bl{X}$ is the blow-up at $C_1\cap C_2$ where $C_1$ has coefficient $-1$ in $C$ and $C_2$ has coefficient $a \neq -1$. One proceeds as in CASE $1$, replacing the factor 
\begin{align*}
\prod_{i=1}^2\frac{\vartheta(C_i+2z)\vartheta(z)}{\vartheta(C_i+z)\vartheta(2z)}
\end{align*}
with the factor
\begin{align*}
\frac{\vartheta(C_1+2z)\vartheta(z)}{\vartheta(C_1+z)\vartheta(2z)}
\frac{\vartheta(C_2-(a+1)z)\vartheta(-z)}{\vartheta(C_1-z)\vartheta((a+1)z)}
\end{align*}
in the definition of $Ell(\Bl{X},\Bl{C};z,\tau)$. By similar computations as above, we get that $Ell_{nv}(\Bl{X},\Bl{C};z,\tau) = Ell_{nv}(X,C;z,\tau)-\frac{\vartheta(z)^2}{\vartheta'(0)^2}\mathrm{Res}_{t=0} F(t)$ where $F(t)$ is given by:
\begin{align*}
&\bigg\{\frac{\vartheta(t+z)\vartheta'(0)}{\vartheta(t)\vartheta(z)}\bigg\}^2\frac{\vartheta(t-z)\vartheta'(0)}{\vartheta(t)\vartheta(-z)}\frac{\vartheta(-t+2z)\vartheta(z)}{\vartheta(-t+z)\vartheta(2z)}\times\\
&\frac{\vartheta(t+(a+1)z)\vartheta(z)}{\vartheta(t+z)\vartheta((a+1)z)}\frac{\vartheta(t-(a+1)z)\vartheta(z)}{\vartheta(t-z)\vartheta((a+1)z)}.
\end{align*}
One easily computes that $\mathrm{Res}_{t=0}F(t) = \frac{\vartheta((a+2)z)\vartheta(a z)}{\vartheta((a+1)z)^2}\frac{\vartheta'(0)^2}{\vartheta(z)^2}$. Hence:
\begin{align*}
Ell_{nv}(\Bl{X},\Bl{C};z,\tau)+\frac{\vartheta((a+2)z)\vartheta(a z)}{\vartheta((a+1)z)^2} = Ell_{nv}(X,C;z,\tau).
\end{align*}
To complete the proof, it therefore suffices to check that the extra summations that occur in definition \ref{Def} differ by $\frac{\vartheta((a+2)z)\vartheta(a z)}{\vartheta((a+1)z)^2}$ as we pass from $\Gamma(C)$ to $\Gamma(\Bl{C})$. There are two cases to consider: either the vertex corresponding to $C_1$ is a bridge, or it is a $-1$ vertex which is not a bridge. In the second case, it is easy to see that the cardinality of $S_a$ simply increases by $1$ as we pass from $\Gamma(C)$ to $\Gamma(\Bl{C})$, which finishes the claim. Otherwise, since the self intersection number of $C_1$ decreases by $1$ after blowing up, the expression $|S_a\backslash R_a|+|R_a|-\sum_{v_i\in B^{'}_a} m_i-1$ remains unchanged, while the sum $\sum_{v_i \in B_a} m_i$ increases by $1$. This completes the proof of the claim.

CASE $3$: $\Bl{X}$ is the blow-up at $C_1\cap C_2$ where $C_i$ has coefficient $a_i\neq -1$ and $a_1+a_2 = -2$. All computations are the same as in CASE $1$, except that now we replace the function $F(t)$ by:
\begin{align*}
\bigg\{\frac{\vartheta(t+z)\vartheta'(0)}{\vartheta(t)\vartheta(z)}\bigg\}^2\frac{\vartheta(t-z)\vartheta'(0)}{\vartheta(t)\vartheta(-z)}\frac{\vartheta(t+2z)\vartheta(z)}{\vartheta(t+z)\vartheta(2z)}
\prod_{i=1}^2\frac{\vartheta(t+(a_i+1)z)\vartheta(z)}{\vartheta(t+z)\vartheta((a_i+1)z)}.
\end{align*}
One easily computes $\mathrm{Res}_{t=0}F(t)=-\mathrm{Res}_{t=-z}F(t) = \prod_{i=1}^2\frac{\vartheta(a_iz)\vartheta'(0)}{\vartheta((1+a_i)z)\vartheta(z)}$.

Since $a_1+a_2 = -2$, the exceptional divisor of the blow-up contributes a bridge vertex $v$ to the graph $\Gamma(\Bl{C})$ with self intersection number $-m_v = -1$. Based on the above calculations, and the formula given in definition \ref{Def}, to complete the claim, it remains to show that the contribution to the elliptic genus:
\begin{align*}
\sum_{a}|S_a \backslash R_a|\frac{\vartheta((a+2)z)\vartheta(a z)}{\vartheta((a+1)z)^2}+\sum_{a}(|R_a|-\sum_{v_i\in B_a} m_i-1)\frac{\vartheta((a+2)z)\vartheta(a z)}{\vartheta((a+1)z)^2}
\end{align*}
is the same for both $\Gamma(C)$ and $\Gamma(\Bl{C})$. The only change in $\Gamma(\Bl{C})$ which could potentially change the above sum is if there exists vertex $w$ with label $a_i$ ($i=1,2$) which is in $S_{a_i}$ as a vertex in $\Gamma(C)$, but in $R_{a_i}$ as a vertex in $\Gamma(\Bl{C})$. In other words, $w$ ends up being connected to the bridge $v$ upon passing to $\Gamma(\Bl{C})$. In this case however, the contribution from $w$ is unchanged, since $m_v - 1 = 0$. Therefore the above sum is the same for both graphs.

CASE $4$: $\Bl{X}$ is the blow-up at a point $p \in C_1$, where $C_1$ has coefficient $-2$. This blow-up creates a special kind of bridge: one adjacent to a single vertex with label $-2$. Again, one proceeds as in CASE $1$, replacing $F(t)$ with:
\begin{align*}
\bigg\{\frac{\vartheta(t+z)\vartheta'(0)}{\vartheta(t)\vartheta(z)}\bigg\}^2\frac{\vartheta(t-z)\vartheta'(0)}{\vartheta(t)\vartheta(-z)}\frac{\vartheta(t+2z)\vartheta(z)}{\vartheta(t+z)\vartheta(2z)}
\frac{\vartheta(t-z)\vartheta(z)}{\vartheta(t+z)\vartheta(-z)}.
\end{align*}
This time $F(t)$ is a meromorphic function on a torus whose only pole occurs at $t=0$. It follows that $\mathrm{Res}_{t=0}F(t) = 0$. Thus, it remains to show that the extra summations associated to the combinatorics of $\Gamma(C)$ remain unchanged as we pass to $\Gamma(\Bl{C})$. But this is evident from the fact that there are no contributions to the summations in definition \ref{Def} coming from $-2$ vertices or their associated bridges, since $\frac{\vartheta((-2+2)z)\vartheta(-2z)}{\vartheta(-z)^2} = 0$.

CASE $5$: $\Bl{X}$ is the blow-up at $C_1\cap C_2$ where $C_i$ has coefficient $a_i\neq -1$ and $a_1+a_2 \neq -2$. One easily verifies, as in the proof of the change of variable formula for the elliptic genus \cite{BL,CLW} that $Ell_{nv}(\Bl{X},\Bl{C};z,\tau) = Ell_{nv}(X,C;z,\tau)$. Therefore, it suffices to show that the extra summations in definition \ref{Def} remain unchanged as we pass from $\Gamma(C)$ to $\Gamma(\Bl{C})$. Let $v_i$ denote the vertices of $\Gamma(C)$ corresponding to $C_i$. Clearly $\Gamma(\Bl{C})$ is obtained from $\Gamma(\Bl{C})$ by connecting $v_1$ and $v_2$ to a vertex $v_3$ with label $a_1+a_2+1$ and genus $0$. Assume $a_i$ are both greater than $-1$. Then $a_1+a_2+1 > a_i$, so the connectivity properties of the vertices of $\Gamma(C)$ are unchanged after inserting $v_3$. The same goes for the case when $a_i < -1$, for in this case $a_1+a_2 + 1 < a_i$. Finally, if the signs of $a_i$ differ, the connectivity properties of $\Gamma(C)$ remain the same after inserting $v_3$. This completes the proof.
\end{proof}

\begin{rmk}\rm
Suppose $Z$ is an $n$-dimensional $\Q$-Gorenstein variety which is not log-terminal. It may happen that $Z$ possesses a log resolution $f:X\rightarrow Z$ such that all discrepancy coefficients $a_i \neq -1$, even though some $a_i < -1$. One might therefore be tempted to define the elliptic genus of $Z$ using the Borisov-Libgober formula. The problem with this approach, as Borisov and Libgober point out in \cite{BLSing}, is proving that the definition is independent of the choice of log resolution. For example, if $Y\rightarrow Z$ is a different log resolution with discrepancy coefficients $b_i \neq -1$, it may happen that the only way to connect $X$ to $Y$ by a sequence of blow-ups and blow-downs (and thereby prove that the two definitions coincide) is to pass through a resolution possessing a $-1$ discrepancy. The above theorem shows (without relying on minimal models), that at least for the case of normal surfaces, such an approach gives an unambiguous definition. 
\end{rmk}

\section{Comparison With Other Invariants}
In this section we compare $Ell(Z;z,\tau)$ to some known invariants of singular normal surfaces.
\subsection{Simple Elliptic Singularities}
We say that $Z$ has a simple elliptic singularity if the singularity is obtained by collapsing a smooth genus $1$ curve to a point. One easily checks that such a singularity is purely log-canonical. Prior to the work in this paper, no definition existed even for the stringy Euler number of $Z$. In what follows, we will show that the ample divisor approach to defining $\widehat{Ell}(Z;z,\tau)$ is well-behaved, and that this definition coincides with $Ell(Z;z,\tau)$. Before proceeding, we will need the following technical lemma:
\begin{lem}\label{hol at zero}
For $j = 1,...,\ell$, let $c_j \in H^2(X)$ and $a_j \neq 0$. Then:
\begin{align*}
\int_X \prod_{i=1}^2\frac{{x_i}\vartheta(\twopi{x_i}-z)}
{\vartheta(\twopi{x_i})}\prod_{j=1}^{\ell}
\frac{\vartheta(\twopi{c_j}-a_j z)\vartheta(z)}
{\vartheta(\twopi{c_j}-z)\vartheta(a_j z)}
\end{align*}
is holomorphic at $z = 0$.
\end{lem}
\begin{proof}
We may rewrite the above expression as:
\begin{align*}
\int_X \prod_{i=1}^2\frac{x_i(1-ye^{-x_i})}{1-e^{-x_i}}
\prod_{j=1}^{\ell}\frac{1-y^{a_j}e^{-c_j}}{1-ye^{-c_j}}
\frac{1-y}{1-y^{a_j}}G(y,\tau)
\end{align*}
Here $y = e^{2\pi i z}$ and $G(y,\tau)$ is a cohomology class whose limit is equal to $1$ as $y \to 1$. By direct computation, one may verify that the integrand in front of $G(y,\tau)$ has a pole of order $1$ at $y=1$ in front of its cohomological degree $2$ term, and no other poles at $y=1$. Since $\lim_{y\to 1}G(y,\tau) = 1$, $G(y,\tau)$ must have a zero at $y=1$ in front of its degree $2$ term. It follows that the limit as $y\to 1$ of the entire integral exists.
\end{proof}

\begin{prop}\label{simple elliptic proof}
Let $Z$ be a Gorenstein projective variety with a simple elliptic singularity. Let $H$ be an ample divisor on $Z$ crossing through the singular point. Define $\widehat{Ell}(Z;z,\tau)$ by taking the limit as $\varepsilon\to 0$ of the perturbed elliptic genus $Ell(Z,-\varepsilon H;z,\tau)$. Then $\widehat{Ell}(Z;z,\tau)$ is well-defined and $\widehat{Ell}(Z;z,\tau) = Ell(Z;z,\tau)$.
\end{prop}

\begin{proof}
Let $\pi: X \rightarrow Z$ be the resolution of singularities obtained by replacing the singular point of $Z$ with an elliptic curve $E$. For simplicity, we may assume that the proper transform of $H$ (which we will also denote by $H$), is a smooth curve which intersects $E$ transversally.

Since $\pi^*K_Z = K_X + E$, we have that $\pi^*(K_Z -\varepsilon H) = K_X + (1-m\varepsilon)E - \varepsilon H$. Therefore ${Ell}(Z,-\varepsilon H;z,\tau)$ is equal to:
\begin{align*}
\int_X \prod_{i=1}^2\frac{{x_i}\vartheta(\twopi{x_i}-z)}
{\vartheta(\twopi{x_i})}\frac{\vartheta(\twopi{E}-m\varepsilon z)
\vartheta(z)}{\vartheta(\twopi{E}-z)\vartheta(m\varepsilon z)}
\frac{\vartheta(\twopi{H}-(1+\varepsilon)z)
\vartheta(z)}{\vartheta(\twopi{H}-z)\vartheta((1+\varepsilon)z)}.
\end{align*}
Here $x_i$ denote the formal Chern roots of $TX$. From the above expression, $Ell(Z,-\varepsilon H;z,\tau)$ is evidently a meromorphic function in the variable $\varepsilon$. Thus, $\lim_{\varepsilon\to 0} Ell(Z,-\varepsilon H;z,\tau)$ exists iff $\lim_{\varepsilon\to 0} \varepsilon Ell(Z,-\varepsilon H;z,\tau) = 0$. By inspection, this second limit is easily seen to be equal to:
\begin{align*}
\frac{1}{m}\int_X \prod_{i=1}^2\frac{{x_i}\vartheta(\twopi{x_i}-z)\vartheta(\twopi{E})}{\vartheta(\twopi{x_i})\vartheta(\twopi{E}-z)}
 = \frac{1}{m}Ell(E;z,\tau) = 0.
\end{align*}
Therefore $\lim_{\varepsilon\to 0} Ell(Z,-\varepsilon H;z,\tau)$ exists. Call this limit $G(z,\tau)$. Define $F(z,\tau) = Ell(Z;z,\tau) - G(z,\tau)$.

CLAIM: $F(z,\tau)$ is holomorphic for $(z,\tau) \in \C\times\mathbb{H}$. We first claim that $Ell(Z;z,\tau)$ is holomorphic. By definition,
\begin{align*}
Ell(Z;z,\tau) = \int_X \prod_{i=1}^2\frac{{x_i}\vartheta(\twopi{x_i}-z)}
{\vartheta(\twopi{x_i})}
\frac{\vartheta(\twopi{E}+2z)\vartheta(z)}
{\vartheta(\twopi{E}+z)\vartheta(2 z)}.
\end{align*}
The possible poles of $Ell(Z;z,\tau)$ occur at the points of the lattice $\frac{1}{2}\Z + \frac{\tau}{2}\Z$. By lemma \ref{hol at zero}, $Ell(Z;z,\tau)$ is holomorphic at $z = 0$. We check that $Ell(Z;z,\tau)$ has no pole at $\frac{\tau}{2}$; the remaining cases are similar. It suffices to show that $\lim_{z\to\frac{\tau}{2}}\vartheta(2z)Ell(Z;z,\tau) = 0$. Evaluating this limit and using the translation rules of the Jacobi theta-function gives the expression:
$$\int_{E} \frac{E\vartheta(\twopi{E}-\frac{\tau}{2})}{\vartheta(\twopi{E})}
\vartheta(\frac{\tau}{2}),$$
which is zero since $E \cong T^2$. Thus, $Ell(Z;z,\tau)$ is holomorphic.
To finish the claim, it therefore suffices to prove that $G(z,\tau)$ is holomorphic in $(z,\tau)$. Fix $\tau \in \mathbb{H}$ and define the function $g(\varepsilon,z) = Ell(Z,-\varepsilon H;z,\tau)\vartheta(m\varepsilon z)\vartheta((1+\varepsilon)z)$. It is clear by lemma \ref{hol at zero} and the transformation properties of theta functions that $g(\varepsilon,z) = \vartheta(z)^2\tilde{g}(\varepsilon,z)$, where $\tilde{g}$ is holomorphic, with $\tilde{g}(0,z) = 0$. Hence 
$$\lim_{\varepsilon\to 0}Ell(Z,-\varepsilon H;z,\tau) = \frac{\vartheta(z)\tilde{g}_{\varepsilon}(0,z)}{mz\vartheta'(0)}$$ 
is holomorphic in $z$.
%The possible poles of $Ell(Z;z,\tau)$ are at $\frac{a}{2}+\frac{b\tau}{2}$ for $a,b$ integers. To prove that $Ell(Z;z,\tau)$ is holomorphic, it suffices to check the points $0, \frac{1}{2}, \frac{\tau}{2}, \frac{1}{2}+\frac{\tau}{2}$.
%The holomorphicity of $Ell(Z;z,\tau)$ at $0$ follows from lemma \ref{hol at zero}. We will now prove that $Ell(Z;z,\tau)$ is holomorphic at $z = \frac{\tau}{2}$. It suffices to show that $\lim_{z\to \frac{\tau}{2}}(z-\frac{\tau}{2})Ell(Z;z,\tau) = 0$. Taking this limit gives us the expression:
%\begin{align*}
%-\frac{e^{\pi i\tau}}{2}\int_X \prod_{i=1}^2\frac{\twopi{x_i}\vartheta(\twopi{x_i}-\frac{\tau}{2})}
%{\vartheta(\twopi{x_i})}\frac{\vartheta(\twopi{E}+\tau)
%\vartheta(\frac{\tau}{2})}{\vartheta(\twopi{E}+\frac{\tau}{2})}.
%\end{align*}
%Using the transformation rule $\vartheta(t+\tau) = -q^{-1/2}e^{-2\pi it}\vartheta(t)$, we have that:
%$$\frac{\vartheta(\twopi{E}+\tau)}
%{\vartheta(\twopi{E}+\frac{\tau}{2})} = e^{-\pi i\tau}\frac{\vartheta(\twopi{E})}{\vartheta(\twopi{E}-\frac{\tau}{2})}.$$
%It follows that $\lim_{z\to \frac{\tau}{2}}(z-\frac{\tau}{2})Ell(Z;z,\tau)$ is equal to $-\frac{1}{2}Ell(T^2;y,q) = 0$. The proofs of the other cases are analogous.

CLAIM: $F(z,\tau)$ is a weak Jacobi form of weight $0$ and index $1$. Let us first examine how $F(z,\tau)$ transforms under the action of $SL_2(\Z)$. It suffices to check the cases for the generators of the $SL_2(\Z)$-action $(z,\tau)\mapsto (z,\tau+1)$ and $(z,\tau)\mapsto (\frac{z}{\tau},-\frac{1}{\tau})$. The identity $F(z,\tau+1) = e^{\pi i/2}F(z,\tau)$ is trivial. 
To verify that $F(\frac{z}{\tau},-\frac{1}{\tau}) = e^{2\pi iz^2/\tau}F(z,\tau)$ we first calculate how $G(z,\tau)$ transforms.
Since we are only interested in the degree $2$ term of the integrand of $Ell(Z;z,\tau,-\varepsilon H)$, we have that: 
\begin{align*}
&Ell(Z;\frac{z}{\tau},\frac{-1}{\tau},-\varepsilon H) = \\
&\int_X \prod_{i=1}^2\frac{x_i\vartheta(\frac{\twopi{x_i}-z}{\tau},-\frac{1}{\tau})}
{\vartheta(\frac{x_i}{2\pi i\tau},-\frac{1}{\tau})}\frac{\vartheta(\frac{\twopi{E}-m\varepsilon z}{\tau},-\frac{1}{\tau})
\vartheta(\frac{z}{\tau},-\frac{1}{\tau})}{\vartheta(\frac{\twopi{E}-z}{\tau},-\frac{1}{\tau})\vartheta(\frac{m\varepsilon z}{\tau},-\frac{1}{\tau})}
\frac{\vartheta(\frac{\twopi{H}-(1+\varepsilon)z}{\tau},-\frac{1}{\tau})
\vartheta(\frac{z}{\tau},-\frac{1}{\tau})}{\vartheta(\frac{\twopi{H}-z}{\tau},-\frac{1}{\tau})\vartheta(\frac{(1+\varepsilon)z}{\tau},-\frac{1}{\tau})}.
\end{align*}
By the transformation formula (\ref{mod two}) for the Jacobi theta function, this expression equals:
\begin{align*}
&e^{2\pi i z^2/\tau}\int_X e^{K_X+(1-m\varepsilon)E-\varepsilon H} \prod_{i=1}^2\frac{{x_i}\vartheta(\twopi{x_i}-z,\tau)}
{\vartheta(\twopi{x_i},\tau)}\frac{\vartheta(\twopi{E}-m\varepsilon z,\tau)
\vartheta(z,\tau)}{\vartheta(\twopi{E}-z,\tau)\vartheta(m\varepsilon z,\tau)}\times\\
&\frac{\vartheta(\twopi{H}-(1+\varepsilon)z,\tau)
\vartheta(z,\tau)}{\vartheta(\twopi{H}-z,\tau)\vartheta((1+\varepsilon)z,\tau)}.
\end{align*}
Expanding $e^{K_X+(1-m\varepsilon)E-\varepsilon H}$ and using $(K_X+(1-m\varepsilon)E-\varepsilon H)E = \pi^*(K_Z-\varepsilon H)E = 0$, we may rewrite the above expression as:
\begin{align*}
&\int_X (1-e^{K_X+(1-m\varepsilon)E-\varepsilon H}) \prod_{i=1}^2\frac{{x_i}\vartheta(\twopi{x_i}-z,\tau)}{\vartheta(\twopi{x_i})}
\frac{\vartheta(\twopi{H}-(1+\varepsilon)z,\tau)
\vartheta(z,\tau)}{\vartheta(\twopi{H}-z,\tau)\vartheta((1+\varepsilon)z,\tau)}\\
&+e^{2\pi i z^2/\tau}Ell(Z;z,\tau,-\varepsilon H)
\end{align*}
Taking the limit as $\varepsilon\to 0$ we get that $G(\frac{t}{\tau},-\frac{1}{\tau})$ is equal to:
\begin{align*}
\int_X (1-e^{K_X+E})\prod_{i=1}^2\frac{{x_i}\vartheta(\twopi{x_i}-z,\tau)}{\vartheta(\twopi{x_i})} + e^{2\pi iz^2/\tau}G(t,\tau).
\end{align*}
Now an analogous calculation shows that $Ell(Z;\frac{z}{\tau},-\frac{1}{\tau})=$
$$\int_X (1-e^{K_X+E})\prod_{i=1}^2\frac{{x_i}\vartheta(\twopi{x_i}-z,\tau)}{\vartheta(\twopi{x_i})} + e^{2\pi iz^2/\tau}Ell(Z;z,\tau).$$
For this case, we are using $(K_X+E)E = K_E[E] = 0$. It follows that $F(\frac{z}{\tau},-\frac{1}{\tau}) = e^{2\pi iz^2/\tau}F(z,\tau)$. The verification of the lattice transformation formula: $F(z+a+b\tau,\tau) = e^{-2\pi i(b^2\tau+2bz)}F(z,\tau)$ proceeds along similar lines and is left to the reader.

It follows that $F(z,\tau)$ transforms as a weak Jacobi form of weight $0$ and index $1$ under the full Jacobi group. Up to a constant, there is only one weak Jacobi form $\phi_{0,1}$ of this weight and index \cite{EZ}. Even without verifying that $F(z,\tau)$ has a Fourier expansion (and hence, is a genuine weak Jacobi form), it is straight-forward to prove that $F(z,\tau) = c\phi_{0,1}(z,\tau)$: For each $\tau$, the quotient $\frac{\phi_{0,1}}{F}$ defines a meromorphic function on the torus $E_\tau$. If $\frac{\phi_{0,1}}{F}$ is non-constant for generic $\tau$, then $\frac{\phi_{0,1}}{F}$ vanishes precisely at the zeros of $\phi_{0,1}$, which are the zeros of the Weierstrass $\mathfrak{p}$ function. Therefore, we must have $\frac{\phi_{0,1}}{F} = \mathfrak{p}(z,\tau)$, which contradicts the fact that $\mathfrak{p}$ has a nontrivial weight under the action of the Jacobi group. Thus, $\frac{\phi_{0,1}}{F}$ is a modular form invariant under the modular group, and is therefore constant.

Finally, one easily computes $\lim_{z\to i\infty}\lim_{\tau\to i\infty}e^{2\pi iz} F(z,\tau) = \chi_0(X)-\chi_0(X) = 0$. However, up to a constant, $\phi_{0,1}$ is the elliptic genus of a $K3$ surface. Since $\lim_{z\to i\infty}\lim_{\tau\to i\infty}e^{2\pi iz}Ell(K3;z,\tau) =\chi_0(K3) = 2$, we must have that $c = 0$. This completes the proof.
\end{proof}

\subsection{Comparison with Veys' stringy $E$-function}
We again recall the definition of Batyrev'stringy $E$-function for $\Q$-Gorenstein varieties $Z$ with log-terminal singularities. Let $f:X\rightarrow Z$ be a log resolution with $K_X = f^*K_Z + \sum_{I} a_i D_i$. The log-terminality condition implies that $a_i > -1$. The stringy $E$-function of $Z$ is then given by the expression:
\begin{align*}
E_{str}(Z;u,v) = \sum_{J\subset I}E(D_{J}^{o};u,v)\prod_{J}
\frac{uv-1}{(uv)^{a_i+1}-1}. 
\end{align*}
Here $D_{J}^{o} = \bigcap_{J}D_j\backslash \bigcup_{J^{c}}D_i$.

We restrict our attention to normal surfaces $Z$. Let $f:X \rightarrow Z$ be a log resolution. Veys made the following observation \cite{Veys}: Suppose the exceptional curve of the log resolution contains a component $C \cong \Proj^1$ which intersects, with multiplicity one, exactly two curves $C_1$ and $C_2$. Let $C_i$ have coefficient $a_i$ and let $C$ have coefficient $a$. Notice first that since $K_X = f^*K_Z + a_1C_1+a_2C_2 + aC +...$, we get, after multiplying both sides by $C$, that $a_1+a_2+2 = m(a+1)$, where $m=-C\cdot C$. Now the contribution from $C$ to the stringy $E$-function of $Z$ is given by:
\begin{align*}
\frac{(L-1)^2}{(L^{a_1+1}-1)(L^{a+1}-1)}+
\frac{(L-1)^2}{(L^{a_2+1}-1)(L^{a+1}-1)}+
(L-1)\frac{(L-1)}{(L^{a+1}-1)}
\end{align*}
where for convenience we have substituted $L$ for $uv$. Using $a_1+a_2+2 = m(a+1)$, this expression simplifies to:
\begin{align*}
&\frac{(L-1)^2(L^{m(a+1)}-1)}{(L^{a+1}-1)(L^{a_1+1}-1)(L^{a_2+1}-1)}
= \frac{(L-1)^2(L^{(m-1)(a+1)}+...+1)}{(L^{a_1+1}-1)(L^{a_2+1}-1)}.
\end{align*}
In particular, the above expression makes sense even when $a=-1$. In this case, the contribution from $C$ to the stringy $E$-function is given by
\begin{align}\label{Veys term}
\frac{m(L-1)^2}{(L^{a_1+1}-1)(L^{a_2+1}-1)}.
\end{align}
A similar analysis applies to the case when $C$ intersects, with multiplicity one a single curve $C_1$ with coefficient $a_1$. In this case, the contribution from $C$ to the stringy $E$-function is obtained from the above formula by setting $a_2 = 0$.

Suppose then that our log resolution $f:X\rightarrow Z$ has the property that every curve $C$ in the exceptional locus $\bigcup_I C_i$ with coefficient $-1$ has genus $0$ and intersects either one or two exceptional curves with multiplicity one. Then, motivated by the above observation, Veys defines the stringy $E$ function of $Z$ to as follows: Let $S \subset I$ denote the subset of indices such that $a_i = -1$. For $i\in S$, let $m_i = -C_i\cdot C_i$ and $a_{i_k}$ ($k=1,2$) the coefficients of the divisors $C_{i_k}$ that intersect $C_i$. Then define the stringy $E$-function $E_{str}(Z;u,v)$ as:
\begin{align*}
\sum_{J\subset I\backslash S}E(C_{J}^{o};u,v)\prod_{J}
\frac{uv-1}{(uv)^{a_i+1}-1} +\sum_{S}
\frac{m_i(uv-1)^2}{((uv)^{a_{i_1}+1}-1)((uv)^{a_{i_2}+1}-1)}.
\end{align*}
We refer to the specialization $E_{str}(Z;u,1)$ as the stringy $\chi_y$ genus of Batyrev-Veys.

Note that from the relation $m_i(a_i+1) = a_{i_1}+a_{i_2}+2$, if $a_i = -1$ then $a_{i_1}+a_{i_2} = -2$. Thus, the vertex corresponding to $C_i$ in the resolution graph of $X$ is a bridge. We may therefore interpret the condition for Veys' stringy $E$-function to be defined as saying that the resolution graph of $X$ contains only bridge $-1$ vertices, which connect at most once to each of their adjacent vertices.
\begin{thm}\label{all equal}
Let $Z$ be a normal surface without strictly log-canonical singularities. Then $\widehat{Ell}(Z;z,\tau)$ has a (possibly not unique) limit for any perturbation of the exceptional divisor of a log resolution. Moreover we can choose a perturbation so that $\widehat{Ell}(Z;z,\tau) = Ell(Z;z,\tau)$. Finally, the stringy $\chi_y$ genus defined by $Ell(Z;z,\tau)$ corresponds to the stringy $\chi_y$ genus of Batyrev-Veys.
\end{thm}

\begin{proof}
By \cite{Veys}, the only $-1$ vertices in the resolution graph of a log resolution of $Z$ are bridges connected with multiplicity one to their adjacent vertices. We first show that the perturbation approach to defining $\widehat{Ell}(Z;z,\tau)$ always gives a limit. Let $f:X\rightarrow Z$ be a log resolution of $Z$ with exceptional divisor $C = \sum a_i C_i$. Introduce a perturbation divisor $C_\varepsilon = \sum \varepsilon b_i C_i$. Then $Ell(X,C+C_\varepsilon;z,\tau)$ is given by:
\begin{align*}
&\int_X\prod_{i=1}^{2}\frac{{x_i}\vartheta(\twopi{x_i}-z)}
{\vartheta(\twopi{x_i})}
\prod_{a_j=-1}\frac{\vartheta(\twopi{C_j}-\varepsilon b_j z)\vartheta(z)}
{\vartheta(\twopi{C_j}-z)\vartheta(\varepsilon b_j z)}\times\\
&\prod_{a_k\neq-1}\frac{\vartheta(\twopi{C_k}-(a_k+1+\varepsilon b_k) z)\vartheta(z)}
{\vartheta(\twopi{C_k}-z)\vartheta((a_k+1+\varepsilon b_k) z)}.
\end{align*}
For each $C_j \cong \Proj^1$ with coefficient $a_j = -1$, let $C_{j_1}, C_{j_2}$ denote the exceptional curves which intersect $C_j$. Then the above integral splits into the following summation:
\begin{align*}
&\sum_{a_j=-1}\int_X\prod_{i=1}^{2}\frac{{x_i}\vartheta(\twopi{x_i}-z)}
{\vartheta(\twopi{x_i})}
\frac{\vartheta(\twopi{C_j}-\varepsilon b_j z)\vartheta(z)}
{\vartheta(\twopi{C_j}-z)\vartheta(\varepsilon b_j z)}\times\\
&\prod_{k=1}^2\frac{\vartheta(\twopi{C_{j_k}}-(a_{j_k}+1+\varepsilon b_{j_k}) z)\vartheta(z)}
{\vartheta(\twopi{C_{j_k}}-z)\vartheta((a_{j_k}+1+\varepsilon b_{j_k}) z)} +...
\end{align*}
Here the $+...$ indicates terms which do not involve curves with $-1$ coefficients, and therefore have well-defined limits as $\varepsilon\to 0$. Multiplying the above expression by $\varepsilon$ and taking the limit as $\varepsilon\to 0$ therefore gives us:
$$\sum_{a_j=-1}\frac{\vartheta(z)}{b_jz\vartheta'(0)}
Ell(\Proj^1,a_{j_1}p_1+a_{j_2}p_2;z,\tau).$$
Since $a_{j_1}+a_{j_2} = -2$, $Ell(\Proj^1,a_{j_1}p_1+a_{j_2}p_2;z,\tau) = 0$. The easiest way to see this is to let $p_1$ and $p_2$ be the north and south poles of $\Proj^1$, and compute $Ell(\Proj^1,a_{j_1}p_1+a_{j_2}p_2;z,\tau)$ using the localization formula, with the standard $S^1$ action on $\Proj^1$. This gives us:
$$\frac{\vartheta(t-(a_{j_1}+1)z)\vartheta(z)}
{\vartheta(t-z)\vartheta((a_{j_1}+1)z)}+
\frac{\vartheta(-t-(a_{j_2}+1)z)\vartheta(z)}
{\vartheta(-t-z)\vartheta((a_{j_2}+1)z)}.$$
This sum clearly evaluates to zero, since $(a_{j_1}+1)=-(a_{j_2}+1)$. It follows that $\lim_{\varepsilon\to 0}Ell(X,C+C_\varepsilon;z,\tau)$ exists.

We now show that $\widehat{Ell}(Z;z,\tau) = Ell(Z;z,\tau)$ for an appropriate choice of a perturbation $C_\varepsilon$. We introduce the following perturbation: Let $C_j\cong \Proj^1$ be a component of $C$ with coefficient $-1$ connected to $C_{j_k}$, $k=1,2$. Let $m_j = -C_j\cdot C_j$. Define $C_{j,\varepsilon}$ to be the perturbation $\varepsilon b_{j_1}C_{j_1}+\varepsilon b_{j_2}C_{j_2}+\varepsilon b_jC_j$, where $b_{j_k}, b_j$ satisfy $m_jb_j= b_{j_1}+b_{j_2}$. Define $C_\varepsilon = \sum_{a_j=-1}C_{j,\varepsilon}$. We take $\widehat{Ell}(Z;z,\tau)$ to be the limit as $\varepsilon\to 0$ of $Ell(X,C+C_\varepsilon;z,\tau)$.

$Ell_{nv}(X,C;z,\tau)$ exhibits the following sum decomposition:
\begin{align*}
\sum_{a_j=-1}\int_X\prod_{i=1}^{2}\frac{{x_i}\vartheta(\twopi{x_i}-z)}
{\vartheta(\twopi{x_i})}
\frac{\vartheta(\twopi{C_j}+2z)\vartheta(z)}
{\vartheta(\twopi{C_j}+z)\vartheta(2 z)}
\prod_{k=1}^2\frac{\vartheta(\twopi{C_{j_k}}-(a_{j_k}+1) z)\vartheta(z)}
{\vartheta(\twopi{C_{j_k}}-z)\vartheta((a_{j_k}+1) z)}+...
\end{align*}
where again $+...$ represents terms that do not involve exceptional components with coefficients $a_i = -1$. It is clear that the terms represented by $+...$ in the expressions for $Ell(X,C+C_\varepsilon;z,\tau)$ and $Ell_{nv}(X,C;z,\tau)$ coincide when we let $\varepsilon\to 0$. Therefore we are reduced to proving that:
\begin{align*}
&\int_X\prod_{i=1}^{2}\frac{{x_i}\vartheta(\twopi{x_i}-z)}
{\vartheta(\twopi{x_i})}
\frac{\vartheta(\twopi{C_j}+2z)\vartheta(z)}
{\vartheta(\twopi{C_j}+z)\vartheta(2 z)}
\prod_{k=1}^2\frac{\vartheta(\twopi{C_{j_k}}-(a_{j_k}+1) z)\vartheta(z)}
{\vartheta(\twopi{C_{j_k}}-z)\vartheta((a_{j_k}+1) z)}+\\
&m_j\frac{\vartheta((a_{j_1}+2)z)\vartheta(a_{j_1}z)}
{\vartheta((a_{j_1}+1)z)^2} =
\lim_{\varepsilon\to 0}\int_X\prod_{i=1}^{2}\frac{{x_i}\vartheta(\twopi{x_i}-z)}
{\vartheta(\twopi{x_i})}
\frac{\vartheta(\twopi{C_j}-\varepsilon b_j z)\vartheta(z)}
{\vartheta(\twopi{C_j}-z)\vartheta(\varepsilon b_j z)}\times\\
&\prod_{k=1}^2\frac{\vartheta(\twopi{C_{j_k}}-(a_{j_k}+1+\varepsilon b_{j_k}) z)\vartheta(z)}
{\vartheta(\twopi{C_{j_k}}-z)\vartheta((a_{j_k}+1+\varepsilon b_{j_k}) z)}.
\end{align*}
In the above equation, expand $\frac{\vartheta(\twopi{C_j}+2z)\vartheta(z)}
{\vartheta(\twopi{C_j}+z)\vartheta(2 z)}$ as $1+C_jG(C_j,z)$, and expand $\frac{\vartheta(\twopi{C_j}-\varepsilon b_j z)\vartheta(z)}
{\vartheta(\twopi{C_j}-z)\vartheta(\varepsilon b_j z)}$ as $1+C_jF(C_j,z)$. The above equation reduces to:
\begin{align}\label{reduced formula}
&\int_{\Proj^1}\frac{{x}\vartheta(\twopi{x}-z)}
{\vartheta(\twopi{x})}\frac{{\nu}\vartheta(\twopi{\nu}-z)}
{\vartheta(\twopi{\nu})}G(\nu,z)
\prod_{k=1}^2\frac{\vartheta(\twopi{p_{j_k}}-(a_{j_k}+1) z)\vartheta(z)}
{\vartheta(\twopi{p_{j_k}}-z)\vartheta((a_{j_k}+1) z)}+
\end{align}
\begin{align*}
&m_j\frac{\vartheta((a_{j_1}+2)z)\vartheta(a_{j_1}z)}
{\vartheta((a_{j_1}+1)z)^2} =
\lim_{\varepsilon\to 0}
\int_{\Proj^1}\frac{{x}\vartheta(\twopi{x}-z)}
{\vartheta(\twopi{x})}\frac{{\nu}\vartheta(\twopi{\nu}-z)}
{\vartheta(\twopi{\nu})}F(\nu,z)\times\\
&\prod_{k=1}^2\frac{\vartheta(\twopi{p_{j_k}}-(a_{j_k}+1+\varepsilon b_{j_k}) z)\vartheta(z)}
{\vartheta(\twopi{p_{j_k}}-z)\vartheta((a_{j_k}+1+\varepsilon b_{j_k}) z)}.
\end{align*}
Here $p_{j_k} = C_j\cap C_{j_k}$, $x = c_1(\Proj^1)$, and $\nu = c_1(\mathcal{O}(-m_j))$. Notice that the above equation depends only on the values of $a_{j_k}, b_j,b_{j_k}$, and $m_j$. Assume first that $m_j = 1$.

Let $C_1$ and $C_2$ be the curves in $\Proj^2 = \set{[x:y:z]}$ given by $x=0$ and $y=0$. Consider the relative elliptic genus $Ell(\Proj^2,a_{j_1}C_1+a_{j_2}C_2;z,\tau)$. Let $\Bl{\Proj^2}$ be the blow-up of $\Proj^2$ at $C_1\cap C_2$ with exceptional divisor $E$. Then 
\begin{align}\label{Proj2 eqn} 
&Ell(\Bl{\Proj^2},a_{j_1}C_1+a_{j_2}C_2-E;z,\tau)+
\frac{\vartheta((a_{j_1}+2)z)\vartheta(a_{j_1}z)}
{\vartheta((a_{j_1}+1)z)^2}=
\end{align}
\begin{align*}
&Ell(\Proj^2,a_{j_1}C_1+a_{j_2}C_2;z,\tau).
\end{align*}
Let $\Delta = b_{j_1}C_1+b_{j_2}C_2$. Consider the relative elliptic genus $Ell(\Proj^2,a_{j_1}C_1+a_{j_2}C_2+\varepsilon\Delta;z,\tau)$. By the change of variable formula for the elliptic genus, this is equal to 
$$Ell(\Bl{\Proj^2},(a_{j_1}+\varepsilon b_{j_1})C_{j_1}+
(a_{j_2}+\varepsilon b_{j_2})C_{j_2}+(\varepsilon b_j-1)E;z,\tau).$$
Clearly taking the limit of the above expression as $\varepsilon\to 0$ reproduces the relative elliptic genus $Ell(\Proj^2,a_{j_1}C_1+a_{j_2}C_2;z,\tau)$. Comparing this limit with the LHS of (\ref{Proj2 eqn}) reduces to the exact same equation as in (\ref{reduced formula}). This proves equation (\ref{reduced formula}) when $m_j=1$. 

Assume now that $m_j > 1$. With $C_1, C_2 \subset \Proj^2$ defined as above, define $\Delta = (b_j-b_{j_2})C_1+b_{j_2}C_2$. Define a sequence of $m_j$ blow-ups of $\Proj^2$ as follows: the first blow-up occurs at $C_1\cap C_2$. Let $E$ be the exceptional divisor. The second blow-up occurs at $\Bl{C}_1\cap E$, where $\Bl{C}_1$ denotes the proper transform of $C_1$. For $i>1$, if $\pi_i: \mathrm{Bl}_i\Proj^2\rightarrow \mathrm{Bl}_{i-1}\Proj^2$ is the $i$-th blow-up, with exceptional divisor $D_i$, then the $(i+1)$-st blow-up occurs at $\Bl{E}\cap D_i$, where $\Bl{E}$ is the proper transform of $E$. It is easy to check that at the $m_j$-th stage, $\Bl{E}$ intersects $\Bl{C}_2$ and $\Bl{D}_{m_j}$ once and that the coefficients of these divisors are $b_{j_2}$ and $b_{j_1}$, respectively. By the change of variable formula, 
$$Ell(\Proj^2,a_{j_1}C_1+a_{j_2}C_2;z,\tau)=
\lim_{\varepsilon\to 0} Ell(\mathrm{Bl}_{m_j}\Proj^2,D;z,\tau)$$
where $D$ is defined by $\pi_{m_j}^*(K_{\Proj_2}-a_{j_1}C_1-a_{j_2}C_2-\varepsilon\Delta)+D = K_{\mathrm{Bl}_{m_j}\Proj^2}$. The proof then follows by applying the same argument as in the $m_j=1$ case to $\mathrm{Bl}_{m_j}\Proj^2$.

Finally we show that the stringy $\chi_y$ genus defined by $Ell(Z;z,\tau)$ coincides with Veys' stringy $\chi_y$ genus. By proposition $3.9$ in \cite{BLSing}, the relative $\chi_y$ 
genus defined by the relation $y\lim_{\tau\to i\infty}Ell(X,D;z,\tau) = \chi_{-y}(X,D)$ coincides with the stringy $\chi_y$ genus of $(X,D)$. 
Consider therefore the relative $\chi_y$ genus $\chi_y(X,C+C_\varepsilon)$ where $C_\varepsilon$ is the perturbation defined above. Let $C_j$ be a divisor with coefficient $-1$ intersecting divisors $C_{j_k}$, $k=1,2$. Then the contribution to $\chi_y(X,C+C_\varepsilon)$ coming from $C_j$ is given by:
$$\frac{(y-1)^2(y^{\varepsilon (b_{j_1}+b_{j_2})}-1)}
{(y^{\varepsilon b_j}-1)(y^{a_{j_1}+\varepsilon b_{j_1}}-1)
(y^{a_{j_2}+\varepsilon b_{j_2}}-1)}.$$
By assumption $b_{j_1}+b_{j_2} = m_j b_j$. Thus, the above contribution simplifies to:
$$\frac{(y-1)^2(y^{\varepsilon b_j(m_j-1)}+...+1)}
{(y^{a_{j_1}+\varepsilon b_{j_1}}-1)
(y^{a_{j_2}+\varepsilon b_{j_2}}-1)}.$$
Taking the limit at $\varepsilon\to 0$ gives precisely the same contribution to the stringy $\chi_y$ genus as in Veys' formula. This completes the proof.
\end{proof}

In case all the bridges of the resolution graph of $Z$ are connected to a single vertex with coefficient $-2$, we may define $\widehat{Ell}(Z;z,\tau)$ by using the perturbation by an ample divisor approach of Borisov and Libgober. In this case we get in addition that $\widehat{Ell}(Z;z,\tau)$ is holomorphic. We may therefore compare $\widehat{Ell}(Z;z,\tau)$ to $Ell(Z;z,\tau)$ using the same argument as in proposition \ref{simple elliptic proof}. All of the essential details are illustrated by the following example:
\begin{ex}\rm
Consider a complete normal surface $Z$ with a singularity of type $x^2+y^3+yz^5$, and with otherwise log-terminal singularities. For example, we could let $Z$ be the singular variety defined by the equation $x_1^2+x_4^3x_2^3+x_2x_3^5$ in $\Proj(3,1,1,1)$. Aside from the above prescribed singularity in the affine neighborhood $x_4 \neq 0$, this space has an $E_8$ (hence canonical) singularity in the neighborhood $x_2\neq 0$. To resolve the $x^2+y^3+yz^5$ singularity, first apply a weighted blow-up at the origin, with weights $(3,2,1)$. The proper transform has an isolated $A_1$-singularity; blowing up one more time gives a resolution of singularities with two $\Proj^1$ exceptional curves. However, it is not a log resolution, since the two exceptional curves intersect at a single point with multiplicity $2$. Blowing up two more times therefore gives a log resolution $\pi:X\rightarrow Z$ with four exceptional curves $E_i \cong \Proj^1$. It is straight-forward to check that $K_X = \pi^*K_Z - E_1-E_2-E_3-2E_4$, with $E_iE_4 = 1$ for $i = 1,2,3$, and $E_iE_j = 0$ for $i,j\neq 4, i\neq j$. The resolution graph of $X$ therefore consists of a vertex for $E_4$ connected by three edges to the three $-1$ vertices corresponding to $E_1, E_2, E_3$.

We now verify that $Ell(Z;z,\tau)$ is holomorphic. First observe that for $p\in \Proj^1$, $Ell(\Proj^1, -2p;z,\tau) = 0$. Indeed, we may assume that $p$ is the north pole of $\Proj^1$ and use the localization formula with the obvious $S^1$ action. We get that $Ell(\Proj^1, -2p;z,\tau) =$
\begin{align*}
\lim_{t\to 0} \frac{\vartheta(t-z)\vartheta(t+z)}{\vartheta(t)\vartheta(t-z)}
\frac{\vartheta(-z)}{\vartheta(z)}+\frac{\vartheta(-t-z)}{\vartheta(-t)} = 0
\end{align*}
It follows that the contribution of $Ell(E_i,-2E_4)\frac{\vartheta(z)}{\vartheta(2z)}$ for $i=1,2,3$ to $Ell(Z;z,\tau)$ is equal to zero. Furthermore, since the coefficient $a_4$ of $E_4$ is equal to $-2$, the contributions $\frac{\vartheta((2+a_4)z)\vartheta(a_4z)}{\vartheta((1+a_4)z)^2}$ to $Ell(Z;z,\tau)$ also vanish.
Therefore, the singular elliptic genus of $Z$ is given by the expression:
\begin{align*}
\int_X \prod_{i=1}^2\frac{{x_i}\vartheta(\twopi{x_i}-z)}
{\vartheta(\twopi{x_i})}\prod_{j=1}^{3}\frac{\vartheta(\twopi{E_j}+2z)\vartheta(z)}{\vartheta(\twopi{E_j}+z)\vartheta(2z)}\cdot\frac{\vartheta(\twopi{E_4}+z)}{\vartheta(\twopi{E_4}-z)}(-1)
\end{align*}
Since the $E_j$ are pair-wise disjoint for $j = 1,2,3$, the above expression may be rewritten as:
\begin{align*}
\sum_{j=1}^3\int_X \prod_{i=1}^2\frac{{x_i}\vartheta(\twopi{x_i}-z)}
{\vartheta(\twopi{x_i})}\frac{\vartheta(\twopi{E_j}+2z)\vartheta(z)}{\vartheta(\twopi{E_j}+z)\vartheta(2z)}\cdot\frac{\vartheta(\twopi{E_4}+z)}{\vartheta(\twopi{E_4}-z)}(-1)
\end{align*}
Clearly the only possible poles for $Ell(Z;z,\tau)$ occur at the half-integral lattice points. We will verify that $Ell(Z;z,\tau)$ has no pole at $z=\frac{\tau}{2}$. As usual, the proofs for the remaining cases are the same. Evaluating the limit $\lim_{z\to\frac{\tau}{2}}\vartheta(2z)Ell(Z;z,\tau)$ gives:
\begin{align*}
&\sum_{j=1}^3\int_X \prod_{i=1}^2\frac{{x_i}\vartheta(\twopi{x_i}-\frac{\tau}{2})}
{\vartheta(\twopi{x_i})}\frac{\vartheta(\twopi{E_j}+\tau)\vartheta(\frac{\tau}{2})}{\vartheta(\twopi{E_j}+\frac{\tau}{2})}\cdot\frac{\vartheta(\twopi{E_4}+\frac{\tau}{2})}{\vartheta(\twopi{E_4}-\frac{\tau}{2})}(-1)\\
&=3e^{-\pi i\tau}\vartheta(\frac{\tau}{2})Ell(\Proj^1,-2p;\frac{\tau}{2},\tau) = 0.
\end{align*}
The triviality of $Ell(\Proj^1,-2p;z,\tau)$ also implies that $\widehat{Ell}(Z;z,\tau)$ is holomorphic. By the same argument as in proposition \ref{simple elliptic proof}, $Ell(Z;z,\tau) = \widehat{Ell}(Z;z,\tau)$.
\end{ex}

\end{document}